\documentclass[]{interact}

\usepackage{epstopdf}
\usepackage[caption=false]{subfig}

\usepackage{mathrsfs,url,color}
\usepackage{graphicx}
\usepackage{amsthm,amsmath,amsfonts,epsfig,wasysym,enumerate,amssymb}
\usepackage{lineno,hyperref}

\usepackage[numbers,sort&compress]{natbib}
\bibpunct[, ]{[}{]}{,}{n}{,}{,}
\renewcommand\bibfont{\fontsize{10}{12}\selectfont}

\theoremstyle{plain}
\newtheorem{theorem}{Theorem}[section]
\newtheorem{lemma}[theorem]{Lemma}
\newtheorem{corollary}[theorem]{Corollary}
\newtheorem{proposition}[theorem]{Proposition}

\theoremstyle{definition}

\newtheorem{example}[theorem]{Example}

\theoremstyle{remark}
\newtheorem{remark}{Remark}

\def\R{\mathbb{R}}

\def\P{\mathbb{P}}
\def\E{\mathbb{E}}
\newcommand{\var}{\operatorname{Var}}		

\begin{document}
			
		\title{On the complete convergence for sequences of dependent random variables via stochastic domination conditions and regularly varying functions theory}
		
	\author{
		\name{Nguyen Chi Dzung\textsuperscript{a} and L\^{e} V\v{a}n Th\`{a}nh\textsuperscript{b}\thanks{CONTACT L\^{e} V\v{a}n Th\`{a}nh. Email: levt@vinhuni.edu.vn}}
		\affil{\textsuperscript{a}Institute of Mathematics, Vietnam Academy of Science and Technology, 18 Hoang Quoc Viet,
			Hanoi 10307, Vietnam\\
			\textsuperscript{b}Department of Mathematics, Vinh University, 182 Le Duan, Vinh, Nghe An, Vietnam}
	}
	
	\maketitle
		
		\begin{abstract}
		This note develops Rio's proof [C. R. Math. Acad. Sci. Paris, 1995] of the rate of convergence in the Marcinkiewicz--Zygmund
	strong law of large numbers to the case of sums of dependent random variables with regularly varying normalizing constants.
It allows us to obtain a complete convergence result for dependent sequences under uniformly bounded moment conditions. This result is new even when the underlying random variables are independent.
	The main theorems are applied to three different dependence structures: (i) $m$-pairwise 
	negatively dependent random variables, (ii) $m$-extended negatively dependent random variables, and (iii) $\varphi$-mixing sequences. 
	To our best knowledge, the results for cases (i) and (ii) are the first results
	in the literature on complete convergence for sequences of $m$-pairwise negatively dependent random variables 
	and $m$-extended negatively dependent random variables 
	under the optimal moment conditions even when $m=1$.
	While the results for cases (i) and (iii) unify and improve many existing ones, the result for case (ii) complements the main result of Chen et al. [J. Appl. Probab., 2010].
	Affirmative answers to open questions raised by Chen et al. [J. Math. Anal. Appl., 2014] and Wu and Rosalsky [Glas. Mat. Ser. III, 2015] are also given.
	An example illustrating the sharpness of the main result is presented.
	\end{abstract}
		

\section{Introduction and motivations}\label{sec:Int}
	
	The maximal inequalities play
	a crucial role in the proofs of the strong law of large numbers (SLLN). 
	Let $\{X,X_n,n\ge1\}$ be a sequence of pairwise independent and identically distributed random variables.
	Etemadi 
\cite{etemadi1981elementary} is the first author who proved the Kolmogorov SLLN 
	\begin{equation*}\label{MZ.01}
		\lim_{n\to\infty}\dfrac{
			\sum_{i=1}^{n}(X_i-\E(X_i))}{n}=0\ \text{ almost surely (a.s.)}
	\end{equation*}
	under optimal moment condition $\E(|X|)<\infty$ without using the
	maximal inequalities. 
	For $1<p<2$, Mart\u{\i}kainen \cite{martikainen1995strong} proved
	that if $\mathbb{E}(|X|^p\log^{\beta}(|X|)) <\infty$
for some $\beta>\max\{0,4p-6\}$,  then the Marcinkiewicz--Zygmund SLLN
	holds, i.e., 
	\begin{equation}\label{MZ.05}
		\lim_{n\to\infty}\dfrac{
			\sum_{i=1}^{n}(X_i-\E(X_i))}{n^{1/p}}=0\ \text{ a.s. }
	\end{equation}
As far as we know, Rio \cite{rio1995vitesses}
	is the first author who proved \eqref{MZ.05} under the optimal moment condition
	$\mathbb{E}(|X|^p)<\infty$. Since then,
	many sub-optimal results on the Marcinkiewicz--Zygmund SLLN have been published. 
	In 2014, Sung \cite{sung2014marcinkiewicz} proposed a different method and
	proved
	\eqref{MZ.05} under a nearly optimal condition $\mathbb{E}(|X|^p(\log\log(|X|))^{2(p-1)} )<\infty$.
	Here and thereafter, $\log(x)$
denotes the natural logarithm (base $e$) of $\max\{x,e\}$, $x\ge0$. 
Very recently, da Silva \cite[Corollary 1]{da2020rates} 
	used the method proposed by Sung \cite{sung2014marcinkiewicz} to prove that if
	$\{X,X_n,n\ge1\}$ are pairwise negatively dependent and identically distributed random variables and $\mathbb{E}(|X|^p)<\infty$, then
	\begin{equation}\label{MZ.07}
		\lim_{n\to\infty}\dfrac{
			\sum_{i=1}^{n}(X_i-\E(X_i))}{n^{1/p}(\log\log(n))^{2(p-1)/p}}=0\ \text{ a.s. }
	\end{equation}
	A very special case of our main result will show that the optimal condition for \eqref{MZ.07} is 
	\begin{equation}\label{MZ.09}
		\mathbb{E}\left(|X|^p/(\log\log(|X|))^{2(p-1)}\right)<\infty.
	\end{equation}
Anh et al. \cite{anh2021marcinkiewicz} recently proved the Marcinkiewicz--Zygmund-type SLLN with the norming constants of the forms $n^{1/p}\tilde{L}(n^{1/p}),\ n\ge1$,
	where $\tilde{L}(\cdot)$ is the Bruijn conjugate of a slowly varying function $L(\cdot)$. However, 
	the proof in \cite{anh2021marcinkiewicz} is based on a maximal inequality for negatively associated random variables 
	which is no longer
	available even for pairwise independent random variables. 
	
	Although Rio's result was extended by the second named author in Th\`{a}nh  \cite{thanh2020theBaum},
	it only considered sums for pairwise independent identically distributed random variables there. 
	The motivation of the present note is that many other dependence structures do not enjoy the Kolmogorov maximal inequality
	such as pairwise negative dependence, extended negative dependence, among others.
	Unlike Th\`{a}nh  \cite{thanh2020theBaum}, we consider
	in this note the case where the underlying sequence of random variables is stochastically dominated by a random variable $X$.
	This allows us to derive the Baum--Katz-type theorem for sequences of dependent random variables satisfying a uniformly bounded moment condition as stated in the following results.
	To our best knowledge, Theorem \ref{thm.main0} and Corollary \ref{cor.main0} are new even when the underlying sequence is comprise 
	of independent random variables. We note that, in Theorem \ref{thm.main0} and Corollary \ref{cor.main0}, 
	no identical distribution condition or stochastic domination condition is assumed.
	
	\begin{theorem}\label{thm.main0}
		Let $1\le p<2$, and $\{X_n,n\ge1\}$
		be a sequence of random variables. Assume that there exists a universal constant $C$ such that
		\begin{equation}\label{eq:bound_var_00}
			\var\left(\sum_{i=k+1}^{k+\ell}f_{i}(X_{i})\right)\le C \sum_{i=k+1}^{k+\ell}\var(f_{i}(X_{i}))
		\end{equation}
		for all $k\ge0,\ell\ge 1$ and all nondecreasing functions $f_i$, $i\ge1$, 
		provided the variances exist.
		Let $L(\cdot)$ be a slowly varying function defined on $[0,\infty)$.
		When $p=1$, we assume further that $L(x)\ge 1$ and is increasing on $[0,\infty)$. 
		If 
		\begin{equation}\label{eq.stoch.domi.13}
			\sup_{n\ge1}\E\left(|X_n|^pL^p(|X_n|)\log(|X_n|)\log^2(\log(|X_n|))\right)<\infty,
		\end{equation}
		then for all $\alpha\ge 1/p$, we have
		\begin{equation}\label{eq.main0.15}
			\sum_{n= 1}^{\infty}
			n^{\alpha p-2}\mathbb{P}\left(\max_{1\le j\le n}\left|
			\sum_{i=1}^{j}(X_i-\E(X_i))\right|>\varepsilon  n^{\alpha}{\tilde{L}}(n^{\alpha})\right)<\infty \text{ for  all } \varepsilon >0,
		\end{equation}
	where $\tilde{L}(\cdot)$ is the Bruijn conjugate of $L(\cdot)$.
	\end{theorem}
	Considering a special interesting case $\alpha=1/p$ and $L(x)\equiv 1$, we obtain the following corollary.
	\begin{corollary}\label{cor.main0}
		Let $1\le p<2$, and $\{X_n,n\ge1\}$
		be a sequence of random variables satisfying condition
		\eqref{eq:bound_var_00}.
		If 
		\begin{equation}\label{eq.stoch.domi.cor.13}
			\sup_{n\ge1}\E\left(|X_n|^p\log(|X_n|)\log^2(\log(|X_n|))\right)<\infty,
		\end{equation}
		then
		\begin{equation}\label{eq.cor.main0.15}
			\sum_{n= 1}^{\infty}
			n^{-1}\mathbb{P}\left(\max_{1\le j\le n}\left|
			\sum_{i=1}^{j}(X_i-\E(X_i))\right|>\varepsilon  n^{1/p}\right)<\infty \text{ for  all } \varepsilon >0.
		\end{equation}
	\end{corollary}
	
	\begin{remark}\label{rem.01}
		\begin{itemize}
			\item [(i)] 
			Since $\{\max_{1\le j\le n}|\sum_{i=1}^j (X_i-\E(X_i))|,n\ge1\}$ is nondecreasing, it follows from \eqref{eq.cor.main0.15} that 
			(see, e.g., Remark 1 in Dedecker and Merlev{\`e}de \cite{dedecker2008convergence})
SLLN \eqref{MZ.05} holds.
			\item[(ii)] For SLLN under the uniformly bounded moment condition, Baxter et al  \cite{baxter2004slln} proved \eqref{MZ.05} with assumptions that the sequence $\{X_n,n\ge1\}$ is independent and $
				\sup_{n\ge1}\E\left(|X_n|^r\right)<\infty \text{ for some } r>p.$ This condition is much stronger than \eqref{eq.stoch.domi.cor.13}.
	Baxter et al. \cite{baxter2004slln} studied the SLLN for weighted sums which is more general than 
	\eqref{MZ.05} but their method does not give the rate of convergence like Corollary \ref{cor.main0}.
			\item[(iii)] For sequence of pairwise independent identically distributed random variables $\{X,X_n,n\ge1\}$, Chen et al. \cite{chen2014bahr}
			obtained \eqref{eq.cor.main0.15} under condition that
			$\mathbb{E}(|X|^p(\log(|X|))^{r})<\infty$  for some $1<p<r<2$. We see that with identical distribution assumption, this moment condition is
still stronger than \eqref{eq.stoch.domi.cor.13}.
			\item[(iv)] Conditions \eqref{eq.stoch.domi.13} and \eqref{eq.stoch.domi.cor.13} are very sharp and almost optimal.
			Even with assumption that the underlying random variables are independent, a special case of
			Example \ref{ex.03} in Section \ref{sec:domination} shows that \eqref{eq.cor.main0.15} may fail
			if  \eqref{eq.stoch.domi.cor.13} is weakened to
			\begin{equation*}\label{eq.stoch.domi.15}
				\sup_{n\ge1}\E\left(|X_n|^p\log(|X_n|)\log(\log(|X_n|))\right)<\infty.
			\end{equation*}
		\end{itemize}
	\end{remark}

The rest of the paper is arranged as follows. Section \ref{sec:complete} presents a complete convergence result for sequences
of dependent random variables with regularly varying normalizing constants.
The proof of Theorem \ref{thm.main0} and an example illustrating the sharpness of the
result are presented in Section \ref{sec:domination}. Finally, Section \ref{sec:appl} contains corollaries and remarks
comparing our results and the ones in the literature.

	\section{Complete convergence for sequences of dependent random variables with regularly varying normalizing constants}\label{sec:complete}
	
	In this section, we will use the method in Rio \cite{rio1995vitesses} to obtain 
	complete convergence for sums
	of dependent random variables with regularly varying constants under stochastic domination condition. 
	The proof is similar to that of Theorem 1 in Th\`{a}nh \cite{thanh2020theBaum}.
	
	A family of random variables $\{X_i,i\in  I\}$ is said to be stochastically dominated by a
	random variable $X$ if
	\begin{equation}\label{eq.stoch.dominated}
		\sup_{i\in I}\P(|X_i|>t)\le \P(|X|>t), \ \text{ for all } t\ge0.
	\end{equation}
	We note that many authors use an apparently weaker definition of $\{X_i,i\in  I\}$ being stochastically dominated
	by a random variable $X$, namely that
	\begin{equation}\label{eq.stoch.domi.06}
		\sup_{i\in I}\P(|X_i|>t)\le C\P(|X|>t), \text{ for all } t\ge0
	\end{equation}	
	for some constant $C\in (0,\infty)$. However, it is shown by Rosalsky and Th\`{a}nh  \cite{rosalsky2021note}
	that \eqref{eq.stoch.dominated} and \eqref{eq.stoch.domi.06} are indeed equivalent.

Let $\rho\in\R$. A real-valued function $R(\cdot )$ is said to be \textit{regularly varying} (at infinity) with index of regular variation
$\rho$ if it is 
a positive and measurable function on $[A,\infty)$ for some $A> 0$, and for each $\lambda>0$,
\begin{equation*}\label{rv01}
	\lim_{x\to\infty}\dfrac{R(\lambda x)}{R(x)}=\lambda^\rho.
\end{equation*}
A regularly varying function with the index of regular variation $\rho=0$ is called \textit{slowly varying} (at infinity).
	If $L(\cdot)$ is a slowly varying function, then by Theorem 1.5.13 in Bingham et al. \cite{bingham1989regular},
	there exists a slowly varying function $\tilde{L}(\cdot)$, unique up to asymptotic equivalence, satisfying
	\begin{equation}\label{BGT1513}
		\lim_{x\to\infty}L(x)\tilde{L}\left(xL(x)\right)=1\ \text{ and } \lim_{x\to\infty}\tilde{L}(x)L\left(x\tilde{L}(x)\right)=1.
	\end{equation}
	The function $\tilde{L}$ is called the de Bruijn conjugate of $L$, and $\left(L,\tilde{L}\right)$ is called a (slowly varying) conjugate pair (see, e.g., 
p. 29 in Bingham et al. \cite{bingham1989regular}).
	If $L(x)=\log^\gamma(x)$ or $L(x)=\log^\gamma\left(\log(x)\right)$ for some $\gamma\in\R$, then $\tilde{L}(x)=1/L(x)$. Especially, if $L(x)\equiv 1$,
	then $\tilde{L}(x)\equiv1$.
	
	Here and thereafter, for a slowly varying function $L(\cdot)$,
	we denote the de Bruijn
	conjugate of $L(\cdot)$
	by $\tilde{L}(\cdot)$. 
	Throughout, we will assume that
	$L(x)$ and $\tilde{L}(x)$ are both continuous on $[0,\infty)$ and
	differentiable on $[A,\infty)$ for some $A>0$. We also assume that (see Lemma 2.2 in Anh et al. \cite{anh2021marcinkiewicz})
	\begin{equation}\label{eq.galambos}
		\lim_{x\to\infty}\dfrac{xL'(x)}{L(x)}=0.
	\end{equation}

	\begin{theorem}\label{thm.main1}
		Let $1\le p<2$, and $\{X_n,n\ge1\}$
		be a sequence of 
		random variables satisfying condition
		\eqref{eq:bound_var_00}.
		Let $L(\cdot)$ be as in Theorem \ref{thm.main0}.
		If $\{X_n, \, n \geq 1\}$ is stochastically dominated by a random variable $X$, and
		\begin{equation}\label{eq.main.13}
			\E\left(|X|^p L^p(|X|)\right)<\infty,
		\end{equation}
		then for all $\alpha\ge 1/p$, we have
		\begin{equation}\label{eq.main.15}
			\sum_{n= 1}^{\infty}
			n^{\alpha p-2}\mathbb{P}\left(\max_{1\le j\le n}\left|
			\sum_{i=1}^{j}(X_i-\E(X_i))\right|>\varepsilon  n^{\alpha}{\tilde{L}}(n^{\alpha})\right)<\infty \text{ for  all } \varepsilon >0.
		\end{equation}
	\end{theorem}
	
We only sketch the proof of Theorem \ref{thm.main1}
and refer the reader to the proof of Theorem 1 in Th\`{a}nh  \cite{thanh2020theBaum} for details.
The main difference here is that we have to consider the nonnegative random
variables so that after applying certain truncation techniques (see \eqref{prop3.4} and \eqref{prop3.3} below), 
the new random variables still satisfy condition \eqref{eq:bound_var_00}.

\begin{proof}[Sketch proof of Theorem \ref{thm.main1}]
Since $\{X_{n}^{+},n\ge1\}$ and $\{X_{n}^{-},n\ge1\}$ satisfy the assumptions of the theorem and
		$X_{n}=X_{n}^{+}-X_{n}^{-},n\ge1$, without loss of generality we can assume that $X_n\ge 0$ for all $n\ge1$.
		For $n\ge1$, set 
		\[b_n=\begin{cases}
			n^{\alpha}\tilde{L}\left(A^{\alpha}\right)& \text{  if } 0\le  n<A,\\
			n^{\alpha}\tilde{L}\left(n^{\alpha}\right)& \text{  if }  n\ge  A,\\
		\end{cases}\]
		\begin{equation}\label{prop3.4}
			X_{i,n}=X_i\mathbf{1}(X_i\le b_n)+b_n\mathbf{1}(X_i> b_n),\ 1\le i\le n,
		\end{equation}
		and
		\begin{equation}\label{prop3.3}
			Y_ {i, m} =\left(X_ {i, 2^m}-X_{i, 2^{m-1}}\right)-\E\left(X_{i, 2^m}-X_{i, 2^{m-1}}\right),\ m\ge1,\ i\ge1.
		\end{equation}
		It is easy to see that $b_n$ is strictly increasing and \eqref{eq.main.15} is equivalent to 
		\begin{equation}\label{prop3.5}
			\sum_{n= 1}^{\infty}
			2^{n(\alpha p-1)}\mathbb{P}\left(\max_{1\le j< 2^n}\left|
			\sum_{i=1}^{j}(X_i-\mathbb{E}(X_i)) \right|>\varepsilon  b_{2^{n}}\right)<\infty \text{ for  all } \varepsilon >0.
		\end{equation}
	It follows from stochastic domination condition
	and definition of $b_{n}$ that
	\begin{equation}\label{prop-15}
		\begin{split}
			0\le \E\left(X_ {i, 2^m}-X_ {i, 2^{m-1}}\right)&\le \E\left(|X|\mathbf{1}(|X|>b_{2^{m-1}})\right).
		\end{split}
	\end{equation}
		Using \eqref{prop-15} and the same argument as in Th\`{a}nh  \cite[Equation (23)]{thanh2020theBaum}, the proof of \eqref{prop3.5} will be completed if we can show that
		\begin{equation}\label{prop-12}
			\sum_{n=1}^{\infty} 2^{n(\alpha p-1)} \mathbb{P}\left(\max_{1\le j< 2^n}\left|\sum_{i=1}^j  
			(X_{i,2^n}-\E (X_{i,2^n}))\right|\ge \varepsilon b_{2^{n-1}} \right)<\infty \text{ for all }\varepsilon>0.
		\end{equation}
		For $m\ge 0,$ set $S_{0,m}=0$ and
		\[S_{j,m}=\sum_{i=1}^j (X_{i,2^m}-\E \left(X_{i,2^m}\right)),\ j\ge 1.\]
		For $1\le j<2^n$ and for $0\le m\le n$, let $k_{j,m}=\lfloor j/2^m\rfloor $ be the greatest integer which is
		less than or equal to $j/2^m$, $j_m = k_{j,m} 2 ^m$. Then (see Th\`{a}nh \cite[Equation (28)]{thanh2020theBaum})
		\begin{equation}\label{prop-18}
			\begin{split}
				\max_{1\le j< 2^n}\left|S_{j,n}\right|
				&\le \sum_{m=1}^n \max_{0\le k<2^{n-m}}\left|\sum_{i=k2^m+1}^{k2^m+2^{m-1}}\left(X_{i, 2^{m-1}}-\E(X_{i, 2^{m-1}})\right)\right| \\
				&\quad +\sum_{m=1}^n \max_{0\le k<2^{n-m}}\left|\sum_{i=k2^m+1}^{(k+1)2^m} Y_{i,m}\right|+\sum_{m=1}^n 2^{m+1}
				\E \left(|X|\mathbf{1}(|X|>b_{2^{m-1}})\right).
			\end{split}
		\end{equation}
Combining \eqref{eq:bound_var_00},  \eqref{prop3.4} and \eqref{prop3.3}, we have for all $m\ge 1$,
		\begin{equation}\label{eq:bound_var1}
			\E\left(\sum_{i=k+1}^{k+\ell} X_{i,2^{m-1}}-\E(X_{i,2^{m-1}})\right)^2\le C\sum_{i=k+1}^{k+\ell} \E(X_{i,2^{m-1}}^2),\ k\ge 0, \ell\ge 1.
		\end{equation}
		and
		\begin{equation}\label{eq:bound_var2}
			\E\left(\sum_{i=k+1}^{k+\ell} Y_{i,m}\right)^2\le C\sum_{i=k+1}^{k+\ell} \E(Y_{i,m}^2),\ k\ge 0, \ell\ge 1.
		\end{equation}
By using \eqref{prop-18}--\eqref{eq:bound_var2}, and the argument as in pages 1236-1238 in Th\`{a}nh  \cite{thanh2020theBaum}, we obtain
		\eqref{prop-12}.
		
	\end{proof}
	
	The next proposition shows that the moment condition in \eqref{eq.main.13} in Theorem \ref{thm.main1} is optimal.
	The proof is the same as that of the implication (iv)$\Rightarrow$(i) of Theorem 3.1 in Anh et al. \cite{anh2021marcinkiewicz}.
	We omit the details.
	
	\begin{proposition}\label{thm.main2}
		Let $1\le p<2$, and let
		$\{X_n, \, n \geq 1\}$ be a sequence of identically distributed
		random variables satisfying \eqref{eq:bound_var_00}, $L(\cdot)$ as in Theorem \ref{thm.main0}. If for some constant $c$, 
		\begin{equation}\label{eq.main.15a}
			\sum_{n= 1}^{\infty}
			n^{-1}\mathbb{P}\left(\max_{1\le j\le n}\left|
			\sum_{i=1}^{j}(X_i-c)\right|>\varepsilon  n^{1/p}{\tilde{L}}(n^{1/p})\right)<\infty \text{ for  all } \varepsilon >0,
		\end{equation}
then
$\E\left(|X_1|^p L^p(|X_1|)\right)<\infty$ and $\E(X_1)=c$.
	\end{proposition}
	\section{On the stochastic domination condition via regularly varying functions}\label{sec:domination}
	
	In this section, we will present a result on the stochastic domination condition via regularly varying functions theory, and use it to prove Theorem \ref{thm.main0}. We need the following simple lemma. See Rosalsky and Th\`{a}nh  \cite{rosalsky2021note} for a proof.
	
	\begin{lemma}\label{lemRT}
		Let $g:[0,\infty)\to [0,\infty)$ be a measurable function with $g(0)=0$ which is
		bounded on $[0,A]$ and differentiable on $[A,\infty)$ for some $A\ge 0$.
		If $\xi$ is a nonnegative random variable, then
		\begin{equation}\label{eq.st.00}
			\begin{split}
				\E(g(\xi))&=\E(g(\xi)\mathbf{1}(\xi\le A))+ g(A)+\int_{A}^\infty g'(x)\P(\xi>x)\mathrm{d} x.
			\end{split}
		\end{equation}
	\end{lemma}
	
	\begin{proposition}\label{prop.sufficiency.for.stochastic.domination.2}
		Let $\{X_i,i\in I\}$ be a family of random variables, and $L(\cdot)$ a 
		slowly varying function.
		If \begin{equation}\label{eq.stoch.domi.12}
			\sup_{i\in I}\E\left(|X_i|^pL(|X_i|)\log(|X_i|)\log^2(\log(|X_i|))\right)<\infty\ \text{ for some }p>0,
		\end{equation}
		then there exists a nonnegative random variable $X$ 
		with distribution function $F(x)=1-\sup_{i\in I}\P(|X_i|>x),\ x\in\R$ 
		such that $\{X_i,i\in I\}$ is stochastically dominated by $X$ and
		\begin{equation}\label{eq.stoch.domi.11}
			\E\left(X^pL(X)\right)<\infty.
		\end{equation}
	\end{proposition}

	\begin{proof} 
		By \eqref{eq.stoch.domi.12} and Theorem 2.5 (i) of Rosalsky and Th\`{a}nh  \cite{rosalsky2021note}, we get that $\{X_i,i\in I\}$ is stochastically dominated by a nonnegative random variable $X$ with distribution function
		\[F(x)=1-\sup_{i\in I}\P(|X_i|>x),\ x\in\R.\]
		Let \[g(x)=x^pL(x)\log(x)\log^2(\log(x)),\ h(x)=x^p L(x),\ x\ge 0.\]
		Applying \eqref{eq.galambos}, there exists $B$ large enough such that $g(\cdot)$ and $h(\cdot)$ are strictly
		increasing on $[B,\infty)$, and
		\[\left|\dfrac{xL'(x)}{L(x)}\right|\le \dfrac{p}{2},\ x>B.\]
		Therefore,
		\begin{equation}\label{eq.st.10}
			h'(x)=px^{p-1}L(x)+x^pL'(x)=x^{p-1}L(x)\left(p+\dfrac{xL'(x)}{L(x)}\right)\le \dfrac{3px^{p-1}L(x)}{2},\ x>B.
		\end{equation}
		By Lemma \ref{lemRT}, \eqref{eq.stoch.domi.12} and  \eqref{eq.st.10}, there exists a constant $C_1$ such that
		\begin{equation*}
			\begin{split}
				\E(h(X))&=\E(h(X)\mathbf{1}(X\le B))+h(B)+\int_{B}^\infty h'(x)\P(X>x)\mathrm{d} x\\
				&\le C_1+\dfrac{3p}{2}\int_{B}^\infty x^{p-1}L(x)\P(X>x)\mathrm{d} x\\
				&= C_1+\dfrac{3p}{2}\int_{B}^\infty x^{p-1}L(x)\sup_{i\in I}\P(|X_i|>x)\mathrm{d} x\\
				&\le C_1+\dfrac{3p}{2}\int_{B}^\infty x^{-1}\log^{-1}(x)\log^{-2}(\log(x))\sup_{i\in I}\E\left(g(|X_i|)\right) \mathrm{d} x\\
				&= C_1+\dfrac{3p}{2}\sup_{i\in I}\E\left(g(|X_i|)\right)\int_{B}^\infty x^{-1}\log^{-1}(x)\log^{-2}(\log(x))\mathrm{d} x\\
				&<\infty.
			\end{split}
		\end{equation*}
The proposition is proved.
	\end{proof}

	\begin{remark}
The contribution of the slowly varying function $L(x)$ in Proposition \ref{prop.sufficiency.for.stochastic.domination.2} help us to unify Theorem 2.5 (ii) and (iii) of Rosalsky and Th\`{a}nh  \cite{rosalsky2021note}. Letting $L(x)=\log^{-1}(x)\log^{-2}(\log(x))$,  $x\ge0$, then by Proposition \ref{prop.sufficiency.for.stochastic.domination.2}, the condition
		\[\sup_{i\in I}\E\left(|X_i|^p\right)<\infty\ \text{ for some }p>0,\]
		implies that the family $\{X_i,i\in I\}$ is stochastically dominated by
		a nonnegative random variable $X$ satisfying
		\[\E\left(X^p\log^{-1}(X)\log^{-2}(\log(X))\right)<\infty. \]
		This slightly improves Theorem 2.5 (ii) in Rosalsky and Th\`{a}nh  \cite{rosalsky2021note}.
		Similarly, by letting $L(x)=1$, we obtain an improvement of Theorem 2.5 (iii) in Rosalsky and Th\`{a}nh  \cite{rosalsky2021note}.
		
	\end{remark}
	
	\begin{proof}[Proof of Theorem \ref{thm.main0}]
		Applying Proposition \ref{prop.sufficiency.for.stochastic.domination.2}, we have from \eqref{eq.stoch.domi.13}
		that the sequence $\{X_n,n\ge1\}$ is stochastically dominated by a nonnegative random variable $X$ with
		\[\E\left(X^pL^p(X)\right)<\infty.\]
	Applying Theorem \ref{thm.main1}, we immediately obtain \eqref{eq.main0.15}.
	\end{proof}

	The following example illustrates the sharpness of Theorem \ref{thm.main0} (and Corollary \ref{cor.main0}).
	It shows that in Theorem \ref{thm.main0}, 
	\eqref{eq.main0.15} may fail if \eqref{eq.stoch.domi.13}
	is weakened to 
	\begin{equation}\label{eq.slln.10}
		\sup_{n\ge1}\E\left(|X_n|^pL^p(|X_n|)\log(|X_n|)\log(\log(|X_n|))\right)<\infty.
	\end{equation}
	
	\begin{example}\label{ex.03}
		Let $1\le p<2$ and $L(\cdot)$ be a positive slowly varying function such that $g(x)=x^pL^p(x)$ is strictly 
		increasing on $[A,\infty)$ for some $A>0$. Let $B=\lfloor  A+ g(A) \rfloor+1$, $h(x)$ be the inverse function of $g(x)$, $x\ge B$, 
		and let $\{X_n,n\ge B\}$ be a sequence of independent random variables such that
		for all $ n\ge B$
		\[\P(X_n=0)=1-\dfrac{1}{n\log(n)\log(\log(n))},\ \P\left(X_n=\pm h(n)\right)=\dfrac{1}{2n\log(n)\log(\log(n))}.\]
		By \eqref{BGT1513}, we can choose (unique up to asymptotic equivalence)
		\[\tilde{L}(x)=\dfrac{h(x^p)}{x},\ x\ge B.\]
		Since $\tilde{L}(\cdot)$ is a slowly varying function,
		\[\log(\tilde{L}(n^{1/p}))=o\left(\log(n)\right),\]
		and so 
		\[\log(h(n))=\log\left(n^{1/p}\tilde{L}(n^{1/p})\right)= \dfrac{1}{p}\log(n)+o(\log(n)).\]
		It thus follows that
		\begin{equation*}\label{eq.stoch.domi.23}
			\begin{split}
				&\sup_{n\ge1}\E\left(|X_n|^pL^p(|X_n|)\log(|X_n|)\log^2(\log(|X_n|))\right)\\
				&=\sup_{n\ge1}\E\left(g(|X_n|)\log(|X_n|)\log^2(\log(|X_n|))\right)\\
				&=\sup_{n\ge1}\dfrac{\log(h(n))\log^2(\log(h(n)))}{\log(n)\log(\log(n))}=\infty,
			\end{split}
		\end{equation*}
		and
		\begin{equation*}\label{eq.stoch.domi.25}
			\begin{split}
				&\sup_{n\ge1}\E\left(|X_n|^pL^p(|X_n|)\log(|X_n|)\log(\log(|X_n|))\right)\\
				&=\sup_{n\ge1}\E\left(g(|X_n|)\log(|X_n|)\log(\log(|X_n|))\right)\\
				&=\sup_{n\ge1}\dfrac{\log(h(n))\log(\log(h(n)))}{\log(n)\log(\log(n))}<\infty.
			\end{split}
		\end{equation*}
		Therefore \eqref{eq.stoch.domi.13} fails but \eqref{eq.slln.10} holds. 
		
		Now, if \eqref{eq.main0.15} holds, then by letting $\alpha=1/p$, we have
		\begin{equation}\label{eq.exm20}
			\lim_{n\to\infty}\dfrac{\sum_{i=B}^n X_i}{n^{1/p}\tilde{L}(n^{1/p})}=0 \text{ a.s.}
		\end{equation}
		It follows from \eqref{eq.exm20} that
		\begin{equation}\label{eq.exm21}
			\lim_{n\to\infty}\dfrac{X_n}{n^{1/p}\tilde{L}(n^{1/p})}=0 \text{ a.s.}
		\end{equation}
		Since the sequence $\{X_n,n\ge1\}$ is comprised of independent random variables, the Borel--Cantelli lemma and \eqref{eq.exm21} ensure that
		\begin{equation}\label{eq.exm23}
			\sum_{n=B}^\infty \P\left(|X_n|>n^{1/p}\tilde{L}(n^{1/p})/2\right)<\infty.
		\end{equation}
		However, we have
		\begin{equation*}\label{eq.exm25}
			\begin{split}
				\sum_{n=B}^\infty \P\left(|X_n|>n^{1/p}\tilde{L}(n^{1/p})/2\right)&=\sum_{n=B}^\infty \P\left(|X_n|>h(n)/2\right)\\
				&=	\sum_{n=B}^\infty\dfrac{1}{n\log(n)\log(\log(n))}=\infty
			\end{split}
		\end{equation*}
		contradicting \eqref{eq.exm23}. Therefore, \eqref{eq.main0.15} must fail.
	\end{example}
	
	\section{Corollaries and remarks}\label{sec:appl}
	
	In this section, we apply Theorems \ref{thm.main0} and \ref{thm.main1} to three different dependence structures: (i) $m$-pairwise 
	negatively dependent random variables, (ii) extended negatively dependent random variables,
	and (iii) $\varphi$-mixing sequences. The results for cases (i) and (ii)
	are new results even $L(x)\equiv1$. We also give remarks to compare our results with the existing ones.
	
	\subsection{$m$-pairwise negatively dependence random variables} 
	The Baum--Katz theorem and the Marcinkiewicz--Zygmund SLLN for sequences of $m$-pairwise
	negatively dependent random variables were studied by Wu and Rosalsky \cite{wu2015strong}.
	Let $m\ge1$ be a fixed integer. A sequence of random
	variables $\{X_n,n\ge1\}$ is said to be \textit{$m$-pairwise negatively dependent} if for all positive integers
	$j$ and $k$ with $|j-k|\ge m$, $X_j$ and $X_k$ are negatively dependent, i.e., 
	\[\mathbb{P}(X\le x, Y\le y)\le \mathbb{P}(X\le x) \mathbb{P}(Y\le y)\text{ for all $x,y \in \mathbb{R}$.}\]
When $m=1$, this reduce to the usual concent of pairwise negative dependence.
		It is well known that if  
	$\{X_i,i\ge1\}$ is a sequence of $m$-pairwise negatively dependent random variables and
	$\{f_i,i\ge1\}$ is a sequence of nondecreasing functions, then 	$\{f_i(X_i),i\ge1\}$ is a sequence  of $m$-pairwise negatively dependent random variables.

The following corollary is the first result
	in the literature on the complete convergence for sequences of $m$-pairwise negatively dependent random variables 
	under the optimal condition even when $m=1$ and $L(x)\equiv1$.
	
	\begin{corollary}\label{prop.pairwise}
		Let $1\le p<2$, $\alpha\ge 1/p$, and let
		$\{X_n, \, n \geq 1\}$ be a sequence of 
		$m$-pairwise negatively dependent random variables, and 
		$L(\cdot)$ as in Theorem \ref{thm.main0}.
\begin{itemize}
	\item[(i)] If \eqref{eq.stoch.domi.13} holds, then we obtain \eqref{eq.main0.15}.
\item[(ii)] If $\{X_n, \, n \geq 1\}$ is stochastically dominated by a random variable $X$ satisfying \eqref{eq.main.13}, then we obtain
\eqref{eq.main.15}.
\end{itemize}
\end{corollary}
	\begin{proof}
From Lemma 2.1 in Wu and Rosalsky \cite{wu2015strong}, it is easy to see that $m$-pairwise negatively dependent random variables
satisfy condition \eqref{eq:bound_var_00}. Corollary \ref{prop.pairwise}
	then follows from Theorems  \ref{thm.main0} and  \ref{thm.main1}.
	\end{proof}

	\begin{remark}\label{rem31} 
		{\rm 
			\begin{description}
	\item[(i)] We consider a special case where $\alpha=1/p$, $1<p<2$ and $L(x)\equiv1$ in Corollary \ref{prop.pairwise}   (ii).  Under the condition
				$\E(|X|^p)<\infty$, we obtain
				\begin{equation}\label{eq.main.15b}
					\sum_{n= 1}^{\infty}
					n^{-1}\mathbb{P}\left(\max_{1\le j\le n}\left|
					\sum_{i=1}^{j}(X_i-\E(X_i))\right|>\varepsilon  n^{1/p}\right)<\infty \text{ for  all } \varepsilon >0,
				\end{equation}
				and
				\begin{equation}\label{eq.main.17b}
					\begin{split}
						\lim_{n\to\infty}\dfrac{\sum_{i=1}^{n}(X_i-\E(X_i))}{n^{1/p}}=0\ \text{ a.s.}
					\end{split}
				\end{equation}
	\item[(ii)]
				For $1<p<2$, Sung \cite{sung2014marcinkiewicz} considered the pairwise independent case and obtained \eqref{eq.main.17b}
				under a slightly stronger condition that 
				\begin{equation}\label{eq.sung2014}
					\mathbb{E}\left(|X|^p(\log\log(|X|))^{2(p-1)}\right)<\infty.
				\end{equation}
				Furthermore, one cannot obtain the rate of convergence \eqref{eq.main.15b} by using the method used in Sung \cite{sung2014marcinkiewicz}. 
				In Chen et al. \cite[Theorem 3.6]{chen2014bahr}, the authors
				proved \eqref{eq.main.15b} holds under condition that $\E(|X|^p\log^r(|X|))<\infty$ for some $r>p$.
They stated an open question whether \eqref{eq.main.15b} holds or not under \eqref{eq.sung2014} (see \cite[Remark 3.1]{chen2014bahr}).
				For the case where the random variables are $m$-pairwise negatively dependent, Wu and Rosalsky  \cite{wu2015strong} obtained \eqref{eq.main.15b}
				and \eqref{eq.main.17b} under condition $\E(|X|^p\log^r(|X|))<\infty$ for some $r>1+p$. Wu and Rosalsky  \cite{wu2015strong} then
				raised an open question that whether \eqref{eq.main.17b} holds or not under Sung's condition \eqref{eq.sung2014}.
				For $p=1$ and also the underlying random variables are $m$-pairwise negatively dependent, Wu and Rosalsky  \cite[Remarks 3.6]{wu2015strong} 
				stated another open question that whether \eqref{eq.main.15b} (with $p=1$) holds or not under the condition $\E(|X|)<\infty$.
				Therefore, a very special case of Corollary \ref{prop.pairwise} gives affirmative answers to the mentioned open questions raised by 
				Chen et al. \cite{chen2014bahr} and Wu and Rosalsky  \cite{wu2015strong}. 
				
\item[(iii)] Let $\alpha=1/p$, $1<p<2$, $L(x)=(\log\log(x))^{2(1-p)/p}$, $x\ge0$. Then $\tilde{L}(x)=(\log\log(x))^{2(p-1)/p}$, $x\ge0$. By Corollary \ref{prop.pairwise} (ii), under 
				condition \eqref{MZ.09}, we obtain \eqref{MZ.07}. Therefore, this special case of Corollary \ref{prop.pairwise} also improves Corollary 1 of da Silva \cite{da2020rates}.
			\end{description}
		}
	\end{remark}
	
	\subsection{Extended negatively dependent random variables}
	The Kolmogorov SLLN for extended negatively dependent was first studied by Chen  et al. \cite{chen2010strong}.
	A collection of 
	random variables $\{X_1,\dots,X_n\}$ is said to be \textit{extended negatively dependent} if for all $x_1,\dots,x_n \in \mathbb{R}$,
	there exists $M>0$ such that
	$$\mathbb{P}(X_1\le x_1, \dots, X_n\le x_n)\le M\mathbb{P}(X_1\le x_1) \dots \mathbb{P}(X_n\le x_n),$$
	and
	$$\mathbb{P}(X_1> x_1, \dots, X_n> x_n)\le M\mathbb{P}(X_1> x_1) \dots \mathbb{P}(X_n> x_n).$$
	A sequence of random variables $\{X_i,i\ge 1\}$ is said to be extended negatively dependent if
	for all $n\ge1$, the collection $\{X_i,1\le i\le n\}$ is extended negatively dependent.
	
Let $m$ be a positive integer. The notion of $m$-extended negative dependence was introduced in Wu and Wang \cite{wu2021strong}. 
A sequence $\{X_i,i\ge 1\}$ of random variables is said to be
$m$-extended negatively dependent if for any $n\ge2$ and any $i_1,i_2,\ldots,i_n$ such that
$|i_j-i_k|\ge m$  for all $1\le j\le k\le n$, we have $\{X_{i_1},\ldots,X_{i_n}\}$
are extended negatively dependent.   If  
$\{X_i,i\ge1\}$ is a sequence  of $m$-extended negatively dependent random variables and
$\{f_i,i\ge1\}$ is a sequence of nondecreasing functions, then 	$\{f_i(X_i),i\ge1\}$ is a sequence  of $m$-extended negatively dependent random variables. 
We note that the classical Kolmogorov maximal inequality or the classical Rosenthal maximal inequality
are not available for extended negatively dependent random variables (see Wu and Wang \cite{wu2021strong}).
	
	\begin{corollary}\label{prop.extended}
		Corollary \ref{prop.pairwise} holds if
		$\{X_n, \, n \geq 1\}$ is a sequence of 
	$m$-extended negatively dependent random variables.
	\end{corollary}
	\begin{proof}
Lemma 3.3 of Wu and Wang \cite{wu2021strong} implies that the sequence $\{X_n, \, n \geq 1\}$ satisfies condition \eqref{eq:bound_var_00}. Corollary \ref{prop.extended}
then follows from Theorems  \ref{thm.main0} and  \ref{thm.main1}.
	\end{proof}
	
	\begin{remark}
		{\rm
			Chen et al. \cite{chen2010strong} proved the Kolmogorov SLLN for
			sequences of extended negatively dependent and identically distributed
			random variables $\{X,X_n,n\ge1\}$ under the condition that $\E(|X|)<\infty$. 
			They used the Etemadi's method in Etemadi \cite{etemadi1981elementary} which does not work for
			the case $1<p<2$ in the Marcinkiewicz--Zygmund SLLN. To our best knowledge, Corollary \ref{prop.extended} is the first result
			in the literature on the Baum--Katz theorem for sequences of $m$-extended negatively dependent random variables under the optimal
			moment condition even when $L(x)\equiv1$ and $m=1$.
		}
	\end{remark}
	
	\subsection{$\varphi$-mixing dependent random variables}

	A sequence of random variables $\{X_n,n\ge 1\}$ is called \textit{$\varphi$-mixing} if
	\[\varphi(n)=\sup_{k\ge 1,A\in \mathcal{F}_{1}^k,B\in \mathcal{F}_{k+n}^\infty,\mathbb{P}(A)>0}\left|\mathbb{P}(B|A)-\mathbb{P}(B)\right|\to 0 \text{ as }n\to\infty,\]
	where $\mathcal{F}_{1}^k=\sigma(X_1,\ldots,X_k)$ and $\mathcal{F}_{k+n}^\infty=\sigma(X_i,i\ge k+n)$.
	The following corollary is the Baum--Katz type theorem for sums of $\varphi$-mixing random variables.
	
	\begin{corollary}\label{prop.phimixing}
		Corollary \ref{prop.pairwise} holds if
		$\{X_n, \, n \geq 1\}$ is a sequence of 
		$\varphi$-mixing random variables satisfying
		\begin{equation}\label{phi-mixing}
			\sum_{n=1}^\infty \varphi^{1/2}(2^n)<\infty.
		\end{equation}
	\end{corollary}
	\begin{proof}
		Corollary 2.3 of Utev \cite{utev1991sums} implies that \eqref{eq:bound_var_00} holds under \eqref{phi-mixing}. Corollary \ref{prop.extended}
		follows from Theorems  \ref{thm.main0} and  \ref{thm.main1}.
	\end{proof}
\begin{remark}		
	The condition \eqref{eq:bound_var_00} is very general.	The concepts of negative association, pairwise negative dependence, $\phi$-mixing can all be extended to
random vectors in Hilbert spaces (see, e.g., \cite{dedecker2008convergence,ko2009note,hien2015weak,hien2019negative}).
It is interesting to see if the results in this note can be extended to dependence random vectors taking values in Hilbert spaces.	
\end{remark}

\textbf{Acknowledgments} The research of the second-named author was supported by the Ministry of Education and Training, grant no. B2022-TDV-01.
	\bibliographystyle{tfs}
\bibliography{mybib}
	
\end{document}